\documentclass{article}

\usepackage{amsmath,amsthm,amssymb,amsfonts}
\usepackage{verbatim}

\usepackage{stmaryrd} 
\usepackage{hyperref}

\usepackage[colorinlistoftodos]{todonotes}

\usepackage{tikz}
\usepackage{tkz-berge, tkz-graph}
\tikzset{vertex/.style={circle,draw,fill,inner sep=0pt,minimum size=1mm}}

\theoremstyle{plain}
\newtheorem{thm}{Theorem}
\newtheorem{lem}[thm]{Lemma}
\newtheorem{prop}[thm]{Proposition}
\newtheorem{cor}[thm]{Corollary}

\theoremstyle{definition}
\newtheorem{definition}[thm]{Definition}
\newtheorem{exl}[thm]{Example}

\numberwithin{thm}{section}

\newcommand{\adj}{\leftrightarrow}
\newcommand{\adjeq}{\leftrightarroweq}

\def\Z{{\mathbb Z}}
\def\N{{\mathbb N}}

\begin{document}
\title{Multivalued Functions in Digital Topology}
\author{Laurence Boxer
         \thanks{
    Department of Computer and Information Sciences,
    Niagara University,
    Niagara University, NY 14109, USA;
    and Department of Computer Science and Engineering,
    State University of New York at Buffalo.
    E-mail: boxer@niagara.edu
    }
}
\date{ }
\maketitle

\begin{abstract}
We study several types of multivalued functions in digital topology.

\end{abstract}

\section{Introduction}
A common method of composing a movie or a
video is via a set of ``frames" or images
that are projected sequentially. If a frame
is an $m \times n$ grid of pixels, then a
transition between projections of consecutive
frames requires reading and then outputting to
the screen each of the $mn$ pixels of the
incoming frame. Since many frames must be
displayed each second (in current technology,
30 per second is common), a video composed
in this fashion uses a lot of memory, and
requires that a lot of data be processed rapidly.

Often, consecutive frames have a great deal of resemblance. In particular, there are
many pairs $(i,j)$ such that the pixel in row~$i$ and column~$j$ is unchanged
between one frame and its successor. If, further, one can efficiently compute for all
changing pixels how they change between successive
frames, then it is not necessary to use as
much storage for the video, as the pixels not
changed between successive frames need not be subjected to the i/o described above; and a larger 
number of frames or their equivalent can be processed per second.
E.g., this approach can be taken in computer
graphics when the changes in the viewer's screen is due to the movements or transformations of sprites.
Thus, the world of applications motivates us to understand the properties of structured single-valued and multivalued
functions between digital images.

Continuous (single-valued and multivalued) functions can often handle changes between
successive frames that seem modeled on continuous Euclidean changes. More general changes may be discontinuous, such as the sudden 
breaking of an object.

In this paper, we develop tools for modeling changes in digital images.
We are particularly concerned with properties of multivalued functions between digital 
images that are characterized by any of the following.
\begin{itemize}
\item Continuity~\cite{egs08,egs12}
\item Weak continuity~\cite{Tsaur}
\item Strong continuity~\cite{Tsaur}
\item Connectivity preservation \cite{Kovalevsky}
\end{itemize}

\section{Preliminaries}
\label{prelims}

Much of this section is quoted or paraphrased from~\cite{BxSt16a} and other
papers cited.

\subsection{Basic notions of digital topology}
We will assume familiarity with the topological theory of digital images. See, e.g.,~\cite{Boxer94} for the standard definitions. All digital images $X$ are assumed to carry their own adjacency relations (which may differ from one image to another). When we wish to emphasize the particular adjacency relation we write the image as $(X,\kappa)$, where $\kappa$ represents
the adjacency relation.

Among the commonly used adjacencies are the $c_u$-adjacencies.
Let $x,y \in \Z^n$, $x \neq y$, where we consider these points as $n$-tuples of integers:
\[ x=(x_1,\ldots, x_n),~~~y=(y_1,\ldots,y_n).
\]
Let $u$ be an integer,
$1 \leq u \leq n$. We say $x$ and $y$ are $c_u$-adjacent if
\begin{itemize}
\item There are at most $u$ indices $i$ for which 
      $|x_i - y_i| = 1$.
\item For all indices $j$ such that $|x_j - y_j| \neq 1$ we
      have $x_j=y_j$.
\end{itemize}
We often label a $c_u$-adjacency by the number of points
adjacent to a given point in $\Z^n$ using this adjacency.
E.g.,
\begin{itemize}
\item In $\Z^1$, $c_1$-adjacency is 2-adjacency.
\item In $\Z^2$, $c_1$-adjacency is 4-adjacency and
      $c_2$-adjacency is 8-adjacency.
\item In $\Z^3$, $c_1$-adjacency is 6-adjacency,
      $c_2$-adjacency is 18-adjacency, and $c_3$-adjacency
      is 26-adjacency.
\item In $\Z^n$, $c_1$-adjacency is $2n$-adjacency and $c_n$-adjacency is $(3^n - 1)$-adjacency.
\end{itemize}

For $\kappa$-adjacent $x,y$, we write $x \adj_{\kappa} y$ or $x \adj y$ when $\kappa$ is understood. We write
$x \adjeq_{\kappa} y$ or $x \adjeq y$ to mean that either $x \adj_{\kappa} y$ or $x = y$.

A subset $Y$ of a digital image $(X,\kappa)$ is
{\em $\kappa$-connected}~\cite{Rosenfeld},
or {\em connected} when $\kappa$
is understood, if for every pair of points $a,b \in Y$ there
exists a sequence $\{y_i\}_{i=0}^m \subset Y$ such that
$a=y_0$, $b=y_m$, and $y_i \adj_{\kappa} y_{i+1}$ for $0 \leq i < m$.
The following generalizes a definition of~\cite{Rosenfeld}.

\begin{definition}\label{continuous}
{\rm ~\cite{Boxer99}}
Let $(X,\kappa)$ and $(Y,\lambda)$ be digital images. A single-valued function
$f: X \rightarrow Y$ is $(\kappa,\lambda)$-continuous if for
every $\kappa$-connected $A \subset X$ we have that
$f(A)$ is a $\lambda$-connected subset of $Y$. 
\end{definition}

When the adjacency relations are understood, we will simply say that $f$ is \emph{continuous}. Continuity can be reformulated in terms of adjacency of points:
\begin{thm}
{\rm ~\cite{Rosenfeld,Boxer99}}
A single-valued function $f:X\to Y$ is continuous if and only if $x \adj x'$ in $X$ implies $f(x) \adjeq f(x')$. \qed
\end{thm}

For two subsets $A,B\subset X$, we will say that $A$ and $B$ are \emph{adjacent} when there exist points $a\in A$ and $b\in B$ such that
$a \adjeq b$. Thus sets with nonempty intersection are automatically adjacent, while disjoint sets may or may not be adjacent.

\subsection{Multivalued functions}
A \emph{multivalued function} $f$ from $X$ to $Y$ assigns a subset of $Y$ to each point of $x$. We will  write $f:X \multimap Y$. For $A \subset X$ and a multivalued function $f:X\multimap Y$, let $f(A) = \bigcup_{x \in a} f(x)$. 

\begin{definition}
\label{mildly}
\rm{\cite{Kovalevsky}}
A multivalued function $f:X\multimap Y$ is \emph{connectivity preserving} if $f(A)\subset Y$ is connected whenever $A\subset X$ is connected.
\end{definition}

As with Definition \ref{continuous}, we can reformulate connectivity preservation in terms of adjacencies.

\begin{thm}
\label{mildadj}
\rm{\cite{BxSt16a}}
A multivalued function $f:X \multimap Y$ is \emph{connectivity preserving} if and only if the following are satisfied:
\begin{itemize}
\item For every $x \in X$, $f(x)$ is a connected subset of $Y$.
\item For any adjacent points $x,x'\in X$, the sets $f(x)$ and $f(x')$ are adjacent. $\Box$
\end{itemize}
\end{thm}

The papers~\cite{egs08, egs12} define continuity for multivalued functions between digital images based on subdivisions. (These papers make an error with respect to compositions, which is corrected in \cite{gs15}.) We have the following.
\begin{definition}
\rm{\cite{egs08, egs12}}
For any positive integer $r$, the \emph{$r$-th subdivision} of $\Z^n$ is
\[ \Z_r^n = \{ (z_1/r, \dots, z_n/r) \mid z_i \in \Z \}. \]
An adjacency relation $\kappa$ on $\Z^n$ naturally induces an adjacency relation (which we also call $\kappa$) on $\Z_r^n$ as follows: $(z_1/r, \dots, z_n/r) \adj_{\kappa} (z'_1/r, \dots, z'_n/r)$ in $\Z^n_r$ if and only if
$(z_1, \dots, z_n) \adj_{\kappa}(z_1, \dots, z_n)$ in $\Z^n$.

Given a digital image $(X,\kappa) \subset (\Z^n,\kappa)$, the \emph{$r$-th subdivision} of $X$ is 
\[ S(X,r) = \{ (x_1,\dots, x_n) \in \Z^n_r \mid (\lfloor x_1 \rfloor, \dots, \lfloor x_n \rfloor) \in X \}. \]

Let $E_r:S(X,r) \to X$ be the natural map sending $(x_1,\dots,x_n) \in S(X,r)$ to $(\lfloor x_1 \rfloor, \dots, \lfloor x_n \rfloor)$. 

For a digital image $(X,\kappa) \subset (\Z^n,\kappa)$, a function $f:S(X,r) \to Y$ \emph{induces a multivalued function $F:X\multimap Y$} as follows:
\[ F(x) = \bigcup_{x' \in E^{-1}_r(x)} \{f(x')\}. \]

A multivalued function $F:X\multimap Y$ is called \emph{continuous} when there is some $r$ such that $F$ is induced by some single valued continuous function $f:S(X,r) \to Y$. 
\end{definition}

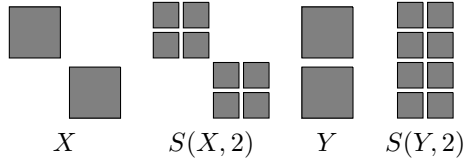
\begin{figure}
\begin{center}
\begin{tabular}{cccc}
\begin{tikzpicture}[scale=.4]
\foreach \x/\y in {1/0,0/1} {
	\filldraw[fill=gray, xshift=2*\x cm,yshift=2*\y cm]
		(45:1.2) \foreach \t in {135,225,315,45} { -- (\t:1.2) };
}
\end{tikzpicture}\qquad
&
\begin{tikzpicture}[scale=.2]
\foreach \x/\y in {2/0,2/1,3/0,3/1,0/2,0/3,1/2,1/3} {
	\filldraw[fill=gray, xshift=2*\x cm,yshift=2*\y cm]
		(45:1.2) \foreach \t in {135,225,315,45} { -- (\t:1.2) };
}
\end{tikzpicture}\qquad
&
\begin{tikzpicture}[scale=.4]
\foreach \x/\y in {0/0,0/1} {
	\filldraw[fill=gray, xshift=2*\x cm,yshift=2*\y cm]
		(45:1.2) \foreach \t in {135,225,315,45} { -- (\t:1.2) };
}
\end{tikzpicture}\qquad
&
\begin{tikzpicture}[scale=.2]
\foreach \x/\y in {0/0,0/1,1/0,1/1,0/2,0/3,1/2,1/3} {
	\filldraw[fill=gray, xshift=2*\x cm,yshift=2*\y cm]
		(45:1.2) \foreach \t in {135,225,315,45} { -- (\t:1.2) };
}
\end{tikzpicture}
\\
$X$ & $S(X,2)$ & $Y$ & $S(Y,2)$
\end{tabular}
\end{center}
\caption{\cite{BxSt16a} Two images $X$ and $Y$ with their second subdivisions. 
\label{subdivfig}}
\end{figure}

Note that in contrast with the definition of continuity, the definition of connectivity preservation makes no reference to $X$ as being embedded inside of any particular integer lattice $\Z^n$.

\begin{exl}
~\cite{BxSt16a}
An example of two spaces and their subdivisions is given in Figure \ref{subdivfig}.
Note that the subdivision construction (and thus the notion of continuity) depends on the particular embedding of $X$ as a subset of $\Z^n$. In particular we may have $X, Y \subset \Z^n$ with $X$ isomorphic to $Y$ but $S(X,r)$ not isomorphic to $S(Y,r)$. This is the case for the two images in Figure~\ref{subdivfig}, when we use 8-adjacency for all images: $X$ and $Y$ in the figure are isomorphic, each being a set of two adjacent points, but $S(X,2)$ and $S(Y,2)$ are not isomorphic since $S(X,2)$ can be disconnected by removing a single point, while this is impossible in $S(Y,2)$. $\Box$
\end{exl}

\begin{lem}
\label{multipleSubdivision}
\rm{\cite{BoxerNP}}
Let $F: (X, c_u) \multimap (Y,c_v)$ be a continuous multivalued function generated by the continuous single-valued function
$f: (S(X,r), c_u) \to (Y,c_v)$. Then for all $n \in \N$ there is a continuous single-valued function
$f': (S(X,nr), c_u) \to (Y,c_v)$ that generates $F$. $\Box$
\end{lem}

\begin{prop}
\label{pt-images-connected}
\rm{\cite{egs08,egs12}}
Let $F:X\multimap Y$ be a continuous multivalued function
between digital images. Then
\begin{itemize}
\item for all $x \in X$, $F(x)$ is connected; and
\item $F$ is connectivity preserving.
\qed
\end{itemize}
\end{prop}

\begin{prop}
\label{1-to-all}
\rm{\cite{BxSt16a}}
Let $X$ and $Y$ be digital images. Suppose $Y$ is connected. Then the
multivalued function $f: X \multimap Y$ defined by
$f(x)=Y$ for all $x \in X$ is connectivity preserving. $\Box$
\end{prop}

\begin{prop}
\label{finite-to-infinite}
\rm{\cite{BxSt16a}}
Let $F: (X,\kappa) \multimap (Y,\lambda)$ be a multivalued
surjection between digital images $(X,\kappa)$ and $(Y,\kappa)$. If $X$ is finite and $Y$
is infinite, then $F$ is not continuous. $\Box$
\end{prop}

\begin{cor}
\rm{\cite{BxSt16a}}
Let $F: X \multimap Y$ be the multivalued function
between digital images defined by
$F(x)=Y$ for all $x \in X$. If $X$ is finite and $Y$ is infinite and connected, then
$F$ is connectivity preserving but not continuous. $\Box$
\end{cor}

Other examples of connectivity preserving but not continuous multivalued functions on finite spaces are given in~\cite{BxSt16a}. 

Other terminology we use includes the following.
Given a digital image $(X,\kappa) \subset \Z^n$ and $x \in X$, the set of points adjacent to $x \in \Z^n$ is
\[N_{\kappa}(x) = \{y \in \Z^n \, | \, y \adj_{\kappa} x\}.\]

\subsection{Weak and strong multivalued continuity}
Other notions of continuity have been given
for multivalued functions between graphs (equivalently,
between digital images). We have the following.

\begin{definition}
\rm{~\cite{Tsaur}}
\label{Tsaur-def}
Let $F: X \multimap Y$ be a multivalued function between
digital images.
\begin{itemize}
\item $F$ has {\em weak continuity}, or is {\em weakly continuous}, if for each pair of
      adjacent $x,y \in X$, $f(x)$ and $f(y)$ are adjacent
      subsets of $Y$.
\item $F$ has {\em strong continuity}, or is {\em strongly continuous}, if for each pair of
      adjacent $x,y \in X$, every point of $f(x)$ is adjacent
      or equal to some point of $f(y)$ and every point of 
      $f(y)$ is adjacent or equal to some point of $f(x)$.
     \qed
\end{itemize}
\end{definition}

Clearly, strong continuity implies weak continuity. Example~\ref{weakNotStrong} below shows that the converse
assertion is false.

\begin{thm}
\label{conn-preserving-char}
\rm{\cite{BxSt16a}}
A multivalued function $F: X \multimap Y$ is connectivity
preserving if and only if the following are satisfied:
\begin{itemize}
\item $F$ has weak continuity.
\item For every $x \in X$, $F(x)$ is a connected subset of $Y$.
$\Box$
\end{itemize}
\end{thm}

\begin{exl}
\label{pt-images-discon}
\rm{\cite{BxSt16a}}
If $F: [0,1]_{\Z} \multimap [0,2]_{\Z}$ is defined by
$F(0)=\{0,2\}$, $F(1)=\{1\}$, then $F$ has both weak and
strong continuity. Thus a multivalued function between
digital images that has weak or strong continuity need not
have connected point-images. By Theorem~\ref{mildadj} and
Proposition~\ref{pt-images-connected} it
follows that neither having weak continuity nor having
strong continuity implies that a multivalued function is
connectivity preserving or continuous.
$\Box$
\end{exl}

\begin{exl}
\rm{\cite{BxSt16a}}
\label{weakNotStrong}
Let $F: [0,1]_{\Z} \multimap [0,2]_{\Z}$ be defined by
$F(0)=\{0,1\}$, $F(1)=\{2\}$. Then $F$ is continuous and
has weak continuity but
does not have strong continuity. $\Box$
\end{exl}

\section{Continuous and other structured multivalued functions}
\label{contAndOthers}
We say a digital image $(X,\kappa)$ has {\em subdivisions preserving adjacency} if
whenever $x \adj_{\kappa} x'$ in $X$, there exist $x_0 \in S(\{x\},r)$ and $x_0' \in S(\{x'\},r)$
such that $x_0 \adj_{\kappa} x_0'$. Images with any $c_u$-adjacency or with the
generalized normal product adjacency~\cite{BoxerNP} $NP_u(c_{u_1},\ldots,c_{u_v})$ on $\Pi_{i=1}^v X_i$,
where $c_{u_i}$ is the adjacency of $X_i$, have subdivisions preserving adjacency.

\begin{thm}
\label{cont-hierarchy}
\rm{\cite{egs08,egs12,BxSt16a}}
Let $(X,\kappa)$ and $(Y,\lambda)$ have subdivisions preserving adjacency.
If $F:(X,\kappa) \multimap (Y,\lambda)$ 
is a continuous multivalued function, then $F$ is connectivity preserving. $\Box$
\end{thm}

\begin{thm}
Let $(X,\kappa)$ and $(Y,\lambda)$ have subdivisions preserving adjacency.
Let $F: (X,\kappa) \multimap (Y,\lambda)$ be continuous. Then $F$ is $(\kappa,\lambda)$-weakly continuous.
\end{thm}

\begin{proof}
Let $f: S(X,r) \to Y$ be a $(\kappa,\lambda)$-continuous function that generates $F$.
Let $x \adj_{\kappa} x'$ in $X$. Then there exist $x_0 \in S(\{x\},r)$ and $x_0' \in S(\{x'\},r)$ such
that $x_0 \adj_{\kappa} x_0'$. Then we have $f(x_0) \in F(x)$, $f(x_0') \in F(x')$, and
$f(x_0) \adjeq_{\lambda} f(x_0')$. Thus, $F(X)$ and $F(x')$ are $\lambda$-adjacent sets, so $F$ is weakly continuous.
\end{proof}

The following example shows that continuity does not imply strong continuity.

\begin{exl}
Let $F: [0,1]_{\Z} \multimap [0,2]_{\Z}$ be defined by $F(0)=\{0\}$, $F(1)=\{1,2\}$. Then $F$ is
$(c_1,c_1)$-continuous but not $(c_1,c_1)$-strongly continuous.
\end{exl}

\begin{proof}
It is easy to see that $F$ is continuous. That $F$ is not strongly continuous follows from the observation that
$0 \adj_{c_1} 1$, but there isn't a member of $F(0)$ that is $c_1$-adjacent to $2 \in F(1)$.
\end{proof}

The following example shows that neither weak continuity nor strong continuity implies continuity.

\begin{exl}
Let $F: [0,1]_{\Z} \multimap [0,2]_{\Z}$ be defined by $F(0)=\{1\}$, $F(1)=\{0,2\}$. Then
$F$ is $(c_1,c_1)$-weakly continuous and $(c_1,c_1)$-strongly continuous but is not $(c_1,c_1)$-continuous.
\end{exl}

\begin{proof}
It is easily seen that $F$ is weakly and strongly continuous. Since $F(1)$ is not $c_1$-connected, by
Theorem~\ref{cont-hierarchy} we can conclude that $F$ is not continuous.
\end{proof}

Examples in~\cite{BxSt16a} show that a connectivity preserving multivalued function between digital images need
not be continuous.

\section{Composition}
\label{compSection}
Suppose $f: (X,\kappa) \multimap (Y,\lambda)$ and $g: (Y, \lambda) \multimap (W,\mu)$ are
multivalued functions between digital images. What properties of $f$ and $g$ are preserved by their
composition? The following are known concerning the
multivalued function $g \circ f: X \multimap W$.
\begin{itemize}
\item \cite{gs15}
      \begin{equation}
      \label{composition-continuity}
      \mbox{If } f \mbox{ and } g \mbox{ are both continuous, } g \circ f \mbox{ need not be continuous.}
      \end{equation}
      There are additional hypotheses explored in~\cite{gs15} under which a composition $g \circ f$ of
      continuous multivalued functions is continuous.
\item \cite{BxSt16a} If $f$ and $g$ are both connectivity preserving, then $g \circ f$ is connectivity
      preserving.
\end{itemize}

We have the following.

\begin{thm}
Let $f: (X,\kappa) \multimap (Y,\lambda)$ and $g: (Y, \lambda) \multimap (W,\mu)$ be
multivalued functions between digital images.
\begin{itemize}
\item If $f$ and $g$ are both weakly continuous, then $g \circ f$ is weakly continuous.
\item If $f$ and $g$ are both strongly continuous, then $g \circ f$ is strongly continuous.
\end{itemize}
\end{thm}

\begin{proof}
Suppose $f$ and $g$ are both weakly continuous. Let $x \adj_{\kappa} x'$ in $X$. Then there exist
$y \in f(x)$ and $y' \in f(x')$ such that $y \adjeq_{\lambda} y'$. Therefore, there exist
$w \in g(y) \subset g \circ f(x)$ and $w' \in g(y') \subset g \circ f(x')$ such that
$w \adjeq_{\mu} w'$. Thus, $g \circ f$ is weakly continuous.

Suppose $f$ and $g$ are both strongly continuous. Let $x \adj_{\kappa} x'$ in $X$. Then for each
$y \in f(x)$ there exists $y' \in f(x')$ such that $y \adjeq_{\lambda} y'$. Then for each
$w \in g(y) \subset g \circ f(x)$ there exists $w' \in g(y') \subset g \circ f(x')$ such that $w \adjeq_{\mu} w'$.
Since $y$ was taken as an arbitrary member of $f(x)$, it follows that for each $w \in g \circ f(x)$ there
exists $w' \in g \circ f(x')$ such that $w \adjeq_{\mu} w'$. Similarly, for each $w' \in g \circ f(x')$ there
exists $w \in g \circ f(x)$ such that $w' \adjeq_{\mu} w$. Thus, $g \circ f$ is strongly continuous.
\end{proof}

\section{Retractions and extensions}
\label{retAndExt}
Retractions and extensions are studied for single-valued continuous functions between digital images
in~\cite{Boxer94}, for continuous multivalued functions in~\cite{egs08,egs12},
and for connectivity preserving multivalued functions between digital images in~\cite{BxSt16a}.
In this section, we obtain more results for retractions and extensions among multivalued functions.

Since we wish to study multivalued functions that have properties of retractions and that
are any of continuous, weakly continuous, strongly continuous, or connectivity preserving,
we will call a multivalued function $F: X \multimap A$ a retraction if $A \subset X$ and for all $a \in A$,
$F(a) = \{a\}$.

\begin{thm}
\label{retAndExtendCont}
Let $A \subset (X, \kappa)$. Then there is a multivalued continuous retraction $R: X \multimap A$ if and only if for
each continuous single-valued function $f: A \to Y$, there is a continuous multivalued function $F: X \multimap Y$ that
extends $f$.
\end{thm}

\begin{proof}
Suppose there is a multivalued continuous retraction $R: X \multimap A$. Then there is a continuous single-valued
function $R': S(X,r) \to A$ that generates $R$. Given a continuous single-valued function $f: A \to Y$,
$f \circ R': S(X,r) \to Y$ is a continuous extension of $f$ and therefore generates a continuous multivalued extension of
$f$ from $X$ to $Y$.

Suppose given any continuous single-valued function $f: A \to Y$ for any $Y$, there is a continuous multivalued function $F: X \multimap Y$ that
extends $f$. Then, in particular, $1_A: A \to A$ extends to a continuous multivalued retraction $R: X \multimap A$.
\end{proof}

The composition of continuous multivalued functions between digital images need not be continuous~\cite{gs15}. However, we have the following (note for $X \subset Z^m$, $c_m$-adjacency is $(3^m - 1)$-adjacency~\cite{Han03}).

\begin{thm}
\label{whenCompositionCont}
\rm{\cite{gs15}}
Let $X \subset (Z^m, 3^m-1)$, $Y \subset (\Z^n, c_u)$ where $1 \le u \le n$, $W \subset (\Z^p,c_v)$ where $1 \le v \le p$. Suppose
$F: X \multimap Y$ and $G: Y \multimap W$ are continuous multivalued functions. Then $G \circ F: (X, 3^m-1) \multimap (W, c_v)$ is a
continuous multivalued function. $\Box$
\end{thm}

This enables us to examine cases for which we can obtain a stronger result than Theorem~\ref{retAndExtendCont}.

\begin{thm}
Let $A \subset X \subset (Z^m, 3^m-1)$. Then there is a multivalued continuous retraction $R: X \multimap A$ if and only if for
each continuous multivalued function $F: (A,3^m-1) \multimap (Y,c_u)$, there is a continuous multivalued function $F': X \multimap Y$ that
extends $F$.
\end{thm}

\begin{proof}
This assertion follows from an argument like that of the proof of Theorem~\ref{retAndExtendCont}, using for one of the
implications that by Theorem~\ref{whenCompositionCont},
$F \circ R$ is a continuous multivalued function.
\end{proof}

We say a multivalued function
$f: X \multimap Y$ is a {\em surjection} 
if for every $y \in Y$ there exists
$x \in X$ such that $y \in f(x)$.

\begin{prop}
\label{surj-prop}
Let $(X,\kappa)$ and $(Y,\lambda)$ be
nonempty digital images. Then there is a
$(\kappa,\lambda)$-strongly continuous multivalued function
$F: X \multimap Y$ that is a surjection.
Further, if $Y$ is connected then there
exists such a multivalued function $F$ that
is also connectivity preserving.
\end{prop}

\begin{proof}
For all $x \in X$, let $F(x)=Y$. Let
$x \adj_{\kappa} x'$ in $X$. Then for
any $y \in F(x)$ we also have $y \in F(x')$, and for
any $y' \in F(x')$ we also have $y' \in F(x)$, so $F$ is strongly continuous.

Clearly, if $Y$ is connected then $F$ is
connectivity preserving.
\end{proof}

\begin{cor}
Let $(X,\kappa)$ and $(Y,\lambda)$ be
nonempty digital images such that the number of $\kappa$-components of $X$ is greater than or equal to
the number of $\lambda$-components of $Y$. Then there is a
$(\kappa,\lambda)$-strongly continuous multivalued function
$F: X \multimap Y$ that is a connectivity preserving surjection.
\end{cor}

\begin{proof}
Let $\{X_u\}_{u \in U}$ be the set of distinct $\kappa$-components of $X$.
Let $\{Y_v\}_{v \in V}$ be the set of distinct $\lambda$-components of $Y$.
Since $|U| \ge |V|$, there is a surjection $s: U \to V$. By Proposition~\ref{surj-prop}, there is a
$(\kappa,\lambda)$-strongly continuous, connectivity preserving surjective multivalued function
$F_u: X_u \multimap Y_{s(u)}$ for every $u \in U$. Then the multivalued function
$F: X \multimap Y$ defined by $F(x)=F_u(x)$ for $x \in X_u$ is easily seen to be a
strongly continuous connectivity preserving surjection.
\end{proof}

\begin{lem}
\label{nbr-nearness-relation}
Let $A \subset (X,\kappa)$, $A \neq \emptyset$, where $X$ is connected. For all
$x \in X$, let $l_X^{\kappa}(x,A)$ be the length
of a shortest $\kappa$-path in $X$ from $x$ to any
point of $A$. For all $x,x' \in X$, if
$x \adj_{\kappa} x'$ then
$|l_X^{\kappa}(x,A) - l_X^{\kappa}(x',A)| \le 1$.
\end{lem}

\begin{proof}
Note the assumption that $X$ is connected
guarantees that for all $x \in X$,
$l_X^{\kappa}(x,A)$ is finite.

Let $\{x_i\}_{i=0}^m$ be a path in $X$
from $x = x_0$ to $x_m \in A$, where
$m=l_X^{\kappa}(x,A)$. Then $\{x'\}\cup \{x_i\}_{i=0}^m$ is a path of length
$m+1$ from $x'$ to $x_m \in A$. Therefore,
$l_X^{\kappa}(x',A) \le m+1 = l_X^{\kappa}(x,A)+1$. 
Similarly, $l_X^{\kappa}(x,A) \le l_X^{\kappa}(x',A)+1$.
The assertion follows.
\end{proof}

In the following, for $x \in X$, $A \subset X$, let $L_A^{\kappa}(x,X)$
be the set such that $a \in L_A^{\kappa}(x,X)$ implies $a \in A$ and there is
a path in $X$ from $x$ to $a$ of length in
$\{l_X^{\kappa}(x,A), l_X^{\kappa}(x,A)+1\}$.

\begin{thm}
\label{constructed-weak}
Let $(X,\kappa)$ be connected.
Suppose $\emptyset \neq A \subset X$. Then the multivalued function
$f: X \multimap A$ defined by
\[ f(x)= \left \{ \begin{array}{ll}
   \{x\} & \mbox{if } x \in A; \\
   L_A^{\kappa}(x,X) & \mbox{if } x \in X \setminus A,
\end{array} \right .
\]
is a weakly continuous retraction.
\end{thm}

\begin{proof}
Let $x \adj_{\kappa} x'$ in $X$. We must consider the following cases.
\begin{itemize}
\item $x,x' \in A$. Then $f(x)=\{x\}$ and 
      $f(x')=\{x'\}$ are clearly adjacent sets.
\item $x \in A$, $x' \in X \setminus A$. 
      Then $l_X^{\kappa}(x',A)=1$ and
      $x \in f(x) \cap f(x')$, so
      $f(x)$ and $f(x')$ are $\kappa$-adjacent sets.
\item $x' \in A$, $x \in X \setminus A$. This is similar to the previous case.
\item $x, x' \in X \setminus A$. Without
      loss of generality,
      \begin{equation}
      \label{nbr-ineq}
      l_X^{\kappa}(x,A) \le l_X^{\kappa}(x',A).
      \end{equation}
      
      Let $a \in A$ be such that there is a $\kappa$-path $P=\{x_i\}_{i=0}^m$ in
      $X$ from $x'=x_0$ to $x_m=a$, such that $m=l_X^{\kappa}(x',A)$.
      By Lemma~\ref{nbr-nearness-relation} and statement~(\ref{nbr-ineq}), $\{x\} \cup P$ is a path in $X$ from $x$ to $a$
      of length at most $L_X^{\kappa}(x,A) + 1$. Therefore, $a \in f(x) \cap f(x')$.
     \end{itemize}
In all cases, $f(x)$ and $f(x')$ are adjacent sets, so $f$ has weak continuity. 

Clearly, $f$ is a retraction.
\end{proof}

The function constructed for
Theorem~\ref{constructed-weak} does not
generally have strong continuity, as
shown in the following example.

\begin{exl}
Let $X=[0,4]_{\Z}$. Let $A=\{0,4\}$. Then
the multivalued function
$f: (X,c_1) \multimap (A,c_1)$ of
Theorem~\ref{constructed-weak}
does not have strong continuity.
\end{exl}

\begin{proof}
This follows from the observations that
$f(1)=\{0\}$, $f(2)=\{0,4\}$, $1 \adj_{c_1} 2$, and $4 \in f(2)$ is not
$c_1$-adjacent to any member of $f(1)$.
\end{proof}

For $A \subset X$, let
$Bd_X^{\kappa}(A) = \{a \in A \, | \, N_X^{\kappa}(a) \setminus A \neq \emptyset \}$.

As an alternative to the multivalued function 
$f$ of Theorem~\ref{constructed-weak}, we can
consider the following.

\begin{thm}
\label{to-bd}
Suppose $\emptyset \neq A \subset (X,\kappa)$. Then the multivalued function
$g: X \multimap A$ defined by
\[ g(x)= \left \{ \begin{array}{ll}
   \{x\} & \mbox{if } x \in A; \\
   Bd_X^{\kappa}(A) & \mbox{if } x \in X \setminus A,
\end{array} \right .
\]
is a weakly continuous retraction. Further, if $Bd_X^{\kappa}(A)$ is connected, then
$g$ is connectivity preserving.
\end{thm}

\begin{proof}
Clearly $g$ is a retraction. To show $g$ is weakly continuous, suppose $x \adj_{\kappa} x'$ in $X$.
We consider the following cases.
\begin{itemize}
\item $x, x' \in A$. Then $x \in g(x)$ and
      $x' \in g(x')$, so $g(x)$ and $g(x')$ are
      $\kappa$-adjacent sets.
\item $x \in A$, $x' \in X \setminus A$. Then
      $x \in Bd_X^{\kappa}(A)$, so
      $x \in g(x) \cap g(x')$. Thus, $g(x)$ and $g(x')$ are
      $\kappa$-adjacent sets.
\item $x' \in A$, $x \in X \setminus A$. This is similar to the previous case.
\item $x,x' \in X \setminus A$. Then $g(x)=g(x')$.
\end{itemize}
Thus, in all cases, $g(x)$ and $g(x')$ are adjacent sets. Therefore, $g$ is weakly continuous.

Suppose $Bd_X^{\kappa}(A)$ is connected. Then for every $x \in X$, $g(x)$ is connected. It
follows from Theorem~\ref{conn-preserving-char} that $g$ is connectivity preserving.
\end{proof}

\begin{exl} Let $X = [0,4]_{\Z} \subset (\Z,c_1)$. Let $A=[1,3]_{\Z} \subset (\Z,c_1)$.
Then the multivalued function $g$ of Theorem~\ref{to-bd} is not strongly continuous.
\end{exl}

\begin{proof}
We have $g(1)=\{1\}$ and $g(0)=\{1,3\}$.
Thus, $0 \adj_{c_1} 1$ but $3 \in g(0)$ has no $c_1$-neighbor in $g(1)$.
\end{proof}

\begin{thm}
\label{retractAndExtend}
\rm{~\cite{Boxer94}} $X_0$ is a (single-valued) retract of $(X,\kappa)$ if and only if
for every digital image $(Y,\lambda)$ and every $(\kappa,\lambda)$-continuous
single-valued function $f: X_0 \to Y$, there exists a $(\kappa,\lambda)$-continuous
extension $f': X \to Y$. $\Box$
\end{thm}

The proof of this assertion depends on the fact that the composition of continuous
single-valued functions is continuous. Since statement~(\ref{composition-continuity}) implies that a similar argument does not
work for multivalued functions, it is not known whether an analog of Theorem~\ref{retractAndExtend}
for multivalued functions is correct.

\begin{thm}
\label{weak-extension}
Let $X_0 \subset (X,\kappa)$ and let $F: X_0 \multimap (Y,\lambda)$ be a
weakly continuous multivalued function. Then there is an extension
$F': X \multimap Y$ that is $(\kappa,\lambda)$-weakly continuous. Further, if
$F$ is connectivity preserving and $F(X_0)$ is $\lambda$-connected, then $F'$ can be 
taken to be connectivity preserving.
\end{thm}

\begin{proof}
Let $F'$ be defined by
\[ F'(x) = \left \{ \begin{array}{ll}
   F(x) & \mbox{if } x \in X_0; \\
   F(X_0) & \mbox{if } x \in X \setminus X_0.
\end{array}
\right .
\]

Let $x \adj_{\kappa} x'$. If $\{x,x'\} \subset X_0$, then $F'(x)=F(x)$ and $F(x')=F'(x')$ are
$\lambda$-adjacent subsets of $Y$. If $\{x,x'\} \subset X \setminus X_0$, then $F'(x)=F(X_0)=F'(x')$.
Otherwise, without loss of generality we have $x \in X_0$ and $x' \in X \setminus X_0$. Then
$F'(x) = F(x) \subset F(X_0) = F'(x')$. Thus, in all cases, $F'(x)$ and $F'(x')$ are $\lambda$-adjacent subsets
of $Y$. Therefore, $F'$ is weakly continuous.

Suppose $F$ is connectivity preserving and $F(X_0)$ is $\lambda$-connected.
Let $A$ be a $\kappa$-connected subset of $X$. If $A \subset X_0$, then $F'(A)=F(A)$ is connected.
Otherwise, $F'(A)=F(X_0)$ is $\lambda$-connected. Thus, $F'$ is connectivity preserving.
\end{proof}

\begin{cor}
Let $X_0 \subset (X,\kappa)$. Then $X_0$ is a weakly continuous multivalued retract of $X$.
\end{cor}

\begin{proof} Since the identity function on $X_0$ is a weakly continuous multivalued function,
the assertion follows from Theorem~\ref{weak-extension}.
\end{proof}

\begin{thm}
\label{cp-preserving-ret}
Let $\emptyset \ne X_0 \subset X$ and let
$F: X_0 \multimap Y$ be a $(\kappa,\lambda)$-connectivity preserving multivalued
function.
If
\begin{itemize}
\item $Y$ is connected, or
\item $F(X_0)$ is connected, or
\item $F(Bd_X^{\kappa}(X_0))$ is connected,
\end{itemize}
then there is a $(\kappa,\lambda)$-connectivity preserving extension $F': X \multimap Y$ of $F$.
\end{thm}

\begin{proof}
Suppose $Y$ is connected. Let $F': X \multimap Y$ be defined by
\[ F'(x) = \left \{ \begin{array}{ll}
   F(x) & \mbox{if } x \in X_0; \\
   Y & \mbox{if } x \in X \setminus X_0.
\end{array}
\right .
\]
Let $A$ be a $\kappa$-connected subset of $X$. If $A \subset X_0$, then
$F'(A)=F(A)$ is connected. Otherwise, $F'(A)=Y$ is connected. Thus, $F'$ is
connectivity preserving.

Suppose $F(X_0)$ is connected. Let $F': X \multimap Y$ be defined by
\[ F'(x) = \left \{ \begin{array}{ll}
   F(x) & \mbox{if } x \in X_0; \\
   F(X_0) & \mbox{if } x \in X \setminus X_0.
\end{array}
\right .
\]
Let $A$ be a $\kappa$-connected subset of $X$. If $A \subset X_0$, then
$F'(A)=F(A)$ is connected. Otherwise, $F'(A)=F(X_0)$ is connected. Thus, $F'$ is
connectivity preserving.

Suppose $Bd_X^{\kappa}(X_0)$ is connected. Let $F': X \multimap Y$ be defined by
\[ F'(x) = \left \{ \begin{array}{ll}
   F(x) & \mbox{if } x \in X_0; \\
   F(Bd_X^{\kappa}(X_0)) & \mbox{if } x \in X \setminus X_0.
\end{array}
\right .
\]
Let $A$ be a $\kappa$-connected subset of $X$.
\begin{itemize}
\item If $A \subset X_0$, then $F'(A)=F(A)$ is connected. 
\item If $A \subset X \setminus X_0$, then $F'(A)=F(Bd_X^{\kappa}(X_0))$ is connected.
\item Otherwise, there exists $a_0 \in A \setminus X_0$ and, for all $a \in A$, a $\kappa$-path $P_a$ in $A$
from $a$ to $a_0$. Then
\begin{equation}
\label{F'(a)}
F'(A)= \bigcup_{a \in A} F'(P_a).
\end{equation}
Since $F'(a_0) \subset F'(P_a)$ for all $a \in A$, if we can show each $F'(P_a)$ is $\lambda$-connected then it will
follow from equation~(\ref{F'(a)}) that $F'(A)$ is $\lambda$-connected.

Suppose $P_a = \{x_i\}_{i=1}^m$ where $x_1=a$, $x_m=a_0$, and $x_{i+1} \adj_{\kappa} x_i$ for $1 \le i < m$. For each
maximal segment $\{x_i\}_{i=u}^v$ of $P_a$ contained in $A \cap X_0$, we have $x_v \in Bd_X^{\kappa}(X_0)$, so
\[ F'(x_v)=F(x_v) \subset F(Bd_X^{\kappa}(X_0)) = F'(x_{v+1}). \]
Thus, $F(\{x_v,x_{v+1}\})$ is $\lambda$-connected. Similarly, if $u>1$ then $F(\{x_{u-1},x_u\})$ is $\lambda$-connected.
It follows that $F(P_a)$ is connected.
\end{itemize}
Thus, $F'$ is connectivity preserving.
\end{proof}

\begin{cor}
\label{cpr}
Let $\emptyset \ne X_0 \subset (X,\kappa)$. If $X_0$ is connected or if $Bd_X^{\kappa}(X_0)$ is connected,
then $X_0$ is a connectivity preserving multivalued retract of $X$.
\end{cor}

\begin{proof}
Since the identity function on $X_0$ is a connectivity preserving multivalued function, the assertion follows
from Theorem~\ref{cp-preserving-ret}.
\end{proof}

An alternate proof of the first assertion of Corollary~\ref{cpr} is given for Theorem~7.2 of~\cite{BxSt16a}.

\section{Wedges}
\label{wedgeSec}
Let $(W,\kappa)$ be a digital image such that $W = X \cup X'$, where $X \cap X' = \{x_0\}$ for some $x_0 \in Y$.
We say $W$ is the {\em wedge} of $X$ and $X'$, written $W = X \wedge X'$ or $(W, \kappa) = (X, \kappa) \wedge (X', \kappa)$.

Let $X \cap X' = \{x_0\}$, $Y \cap Y' = \{y_0\}$, $F: X \multimap Y$ and $F': X' \multimap Y'$ be
multivalued functions such that $F(x_0) = \{y_0\} = F'(x_0)$. Define
$F \wedge F': X \wedge X' \multimap Y \wedge Y'$ by
\[ (F \wedge F')(a) = \left \{ \begin{array}{ll}
    F(a) & \mbox{if } a \in X \setminus \{x_0\}; \\
    F'(a) & \mbox{if } a \in X' \setminus \{x_0\}; \\
    \{y_0\} & \mbox{if } a=x_0.    
\end{array}
\right .
\]

\begin{lem}
\label{piecesInWedge}
Let $(W,\kappa) = (X, \kappa) \wedge (X',\kappa)$, where $\{x_0\}=X \cap X'$. Let $A \subset W$. Then $A$ is $\kappa$-connected if and only 
if each of $A \cap X$ and $A \cap X'$ is $\kappa$-connected.
\end{lem}

\begin{proof}
Suppose $A$ is $\kappa$-connected. Let $a_0,a_1 \in A \cap X$. Then there is a $\kappa$-path $P=\{x_i\}_{i=1}^n$
in $A$ from $a_0$ to $a_1$. If $P$ is not a subset of $A \cap X$ then there are smallest and largest indices $u,v$ such
that $\{x_u, x_v\} \subset (A \cap X') \setminus \{x_0\}$. Then $P_0=\{x_i\}_{i=1}^{u-1}$ is a path in $A \cap X$ from
$a_0$ to $x_0$, and $P_1=\{x_i\}_{i=v+1}^n$ is a path in $A \cap X$ from $x_0$ to $a_1$. Therefore,
$P_0 \cup P_1$ is a path in $A \cap X$ from $a_0$ to $a_1$.
Since $a_0$ and $a_1$ were arbitrarily chosen, $A \cap X$ is connected. Similarly, $A \cap X'$ is connected.

Suppose each of $A \cap X$ and $A \cap X'$ is $\kappa$-connected. Let $a_0,a_1 \in A$.
If $a_0 \in A \cap X$ then there is a path $P$ in $A \cap X$ from $a_0$ to $x_0$. Similarly, if
$a_0 \in A \cap X'$ then there is a path $P$ in $A \cap X'$ from $a_0$ to $x_0$. In either case, there is a path $P$ in $A$ from
$a_0$ to $x_0$. Similarly, there is a path $P'$ in $A$ from $x_0$ to $a_1$. Therefore, $P \cup P'$ is a path in $A$ from
$a_0$ to $a_1$. Since $a_0$ and $a_1$ were arbitrarily chosen, it follows that $A$ is connected.
\end{proof}

\begin{thm}
Let $X \cap X' = \{x_0\}$, $Y \cap Y' = \{y_0\}$. Let $F: X \multimap Y$ and $F': X' \multimap Y'$ be
multivalued functions such that $\{y_0\} = F(x_0) = F'(x_0)$.
\begin{enumerate}
\item If $F$ and $F'$ are both $(\kappa, \lambda)$-continuous, then $F \wedge F': X \wedge X' \multimap Y \wedge Y'$ is
      $(\kappa, \lambda)$-continuous.
\item If $F$ and $F'$ are both $(\kappa, \lambda)$-connectivity preserving, then $F \wedge F': X \wedge X' \multimap Y \wedge Y'$ is
      $(\kappa, \lambda)$-connectivity preserving.
\item If $F$ and $F'$ are both $(\kappa, \lambda)$-weakly continuous, then $F \wedge F': X \wedge X' \multimap Y \wedge Y'$ is
      $(\kappa, \lambda)$-weakly continuous.
\item If $F$ and $F'$ are both $(\kappa, \lambda)$-strongly continuous, then $F \wedge F': X \wedge X' \multimap Y \wedge Y'$ is
      $(\kappa, \lambda)$-strongly continuous.
\end{enumerate}
\end{thm}

\begin{proof} We argue as follows.
\begin{enumerate}
\item Suppose $F$ and $F'$ are both $(\kappa, \lambda)$-continuous.
From Lemma~\ref{multipleSubdivision}, we conclude that for some $r \in \N$, there are continuous functions $f: S(X,r) \to Y$
and $f': S(X',r) \to Y'$ that generate $F$ and $F'$, respectively. Then the single-valued function
$\hat{f}: S(X,r) \cup S(X',r) = S(X \wedge X',r) \to Y \wedge Y'$ defined by
\[ \hat{f}(a) = \left \{ \begin{array}{ll}
f(a) & \mbox{if } a \in S(X,r); \\
f'(a) & \mbox{if } a \in S(X',r)
\end{array}
\right .
\]
is well defined since $F(x_0)=\{y_0\}=F'(x_0)$, and is continuous and generates $F \wedge F'$.
Thus, $F \wedge F'$ is continuous.
\item Suppose $F$ and $F'$ are both $(\kappa, \lambda)$-continuity preserving. Let $A$ be a connected subset of
      $X \wedge X'$. If $A \subset X$ or if $A \subset X'$ then $(F \wedge F')(A)$ is either $F(A)$ or $F'(A)$, hence is connected.
      Otherwise, by Lemma~\ref{piecesInWedge}, each of $A \cap X$ and $A \cap X'$ is connected. Therefore,
      $F(A \cap X)$ and $F'(A \cap X')$ are connected, and $F(A \cap X) \cap F'(A \cap X')= \{y_0\}$.
      Thus, $(F \wedge F')(A)=F(A \cap X) \cup F'(A \cap X')$ is connected. Hence
      $F \wedge F'$ is connectivity preserving.
\item Suppose $F$ and $F'$ are both $(\kappa, \lambda)$-weakly continuous. Let $a_0 \adj_{\kappa} a_1$ in $X \wedge X'$. Then
      either $\{a_0,a_1\} \subset X$ or $\{a_0,a_1\} \subset X'$. Therefore either
      $(F \wedge F')(\{a_0,a_1\})=F(\{a_0,a_1\})$ or $(F \wedge F')(\{a_0,a_1\})=F'(\{a_0,a_1\})$. In either case,
      $(F \wedge F')(a_0)$ and $(F \wedge F')(a_1\})$ are adjacent sets.
      Thus, $(F \wedge F')$ is weakly continuous.
\item Suppose $F$ and $F'$ are both $(\kappa, \lambda)$-strongly continuous. Let $a_0 \adj_{\kappa} a_1$ in $X \wedge X'$. As above
      for our argument concerning weak continuity, either
      $(F \wedge F')(\{a_0,a_1\})=F(\{a_0,a_1\})$ or $(F \wedge F')(\{a_0,a_1\})=F'(\{a_0,a_1\})$. In either case, one sees from
      Definition~\ref{Tsaur-def} that $F \wedge F'$ is strongly continuous.
\end{enumerate}
\end{proof}

\section{Further remarks}
We have studied properties of structured multivalued functions between digital images.
In section~\ref{contAndOthers}, we studied relations between continuous multivalued functions and other structured types
of multivalued functions. In section~\ref{compSection}, we studied properties of multivalued functions that are preserved by
composition. In section~\ref{retAndExt}, we studied retractions and extensions of structured multivalued functions. In
section~\ref{wedgeSec}, we studied properties of multivalued functions that are preserved by the wedge operation.

\end{document}